\documentclass[12pt]{article}

\usepackage{amsthm,amsfonts,amsbsy,amssymb,amsmath,amscd}
\usepackage[all]{xy}
\usepackage[colorlinks=true]{hyperref}

\theoremstyle{plain}
\newtheorem{theorem}{Theorem}[section]
\newtheorem{lemma}[theorem]{Lemma}

\newtheorem{prop}[theorem]{Proposition}

\theoremstyle{remark}
\newtheorem{remark}[theorem]{Remark}%[section]

\theoremstyle{plain}
\theoremstyle{plain}
\theoremstyle{plain}
\theoremstyle{plain}
\theoremstyle{plain}
\theoremstyle{plain}

\newcommand{\lb}{\overline{\mathcal{L}}}   %arithmetic line bundle

\newcommand{\ob}{\overline{\mathcal{O}}}
\newcommand{\eb}{\overline{\mathcal{E}}}
\newcommand{\fb}{\overline{\mathcal{F}}}

\newcommand{\ab}{\overline{\mathcal{A}}}
\newcommand{\mb}{\overline{\mathcal{M}}}

\newcommand{\CL}{{\mathcal{L}}}
\newcommand{\CO}{{\mathcal{O}}}
\newcommand{\CE}{{\mathcal{E}}}
\newcommand{\CX}{{\mathcal{X}}}

\newcommand{\CM}{{\mathcal{M}}}

\newcommand{\mbar}{\overline{M}}            %normed module

\newcommand{\CC}{{\mathbb{C}}}
\newcommand{\QQ}{{\mathbb{Q}}}
\newcommand{\RR}{{\mathbb{R}}}
\newcommand{\ZZ}{{\mathbb{Z}}}
\newcommand{\hhat}{\widehat h^0}                %arithmetic global section
\newcommand{\Hhat}{\widehat H^0}
\newcommand{\hsefhat}{\widehat h^0_{\rm sef}}
\newcommand{\Hsefhat}{\widehat H^0_{\rm sef}}

\newcommand{\Spec}{\mathrm{Spec}}

\newcommand{\vol}{\mathrm{vol}}             % volume
\newcommand{\volhat}{\widehat{\mathrm{vol}}}             % arithmetic volume
\newcommand{\dvol}{\mathrm{dvol}}             % derivative of volume

\newcommand{\deghat}{\widehat{\deg}}

\newcommand{\Pichat}{\widehat{\mathrm{Pic}}}
\newcommand{\PPichat}{\widehat{\mathcal{P}{ic}}}
\newcommand{\NNefhat}{\widehat{\mathcal{N}{ef}}}
\newcommand{\BBighat}{\widehat{\mathcal{B}{ig}}}
\newcommand{\EEffhat}{\widehat{\mathcal{E}{ff}}}

\begin{document}
\title{Effective Bounds of Linear Series
on Algebraic Varieties and Arithmetic Varieties}
\author{Xinyi Yuan, Tong Zhang}

%\address{Department of Mathematics, University of California, Berkeley, CA 94720, U.S.A.}
%\email{yxy@math.berkeley.edu}

%\address{Department of Mathematics, University of Alberta, Edmonton, Alberta T6G 2G1, Canada}
%\email{tzhang5@ualberta.ca}

\maketitle

\tableofcontents

\section{Introduction}

In this paper, we prove effective upper bounds for effective sections of line bundles on projective varieties and hermitian line bundles on arithmetic varieties in terms of intersection numbers. They are effective versions of the Hilbert--Samuel formula and the arithmetic Hilbert--Samuel formula. The treatments are high-dimensional generalizations of \cite{YZ1, YZ2}. Similar results are obtained independently by Huayi Chen \cite{Ch4} with less explicit error terms. 

The initial motivation for our first paper \cite{YZ1} is to obtain some arithmetic version of the classical Noether inequality on minimal surfaces. We have achieved the goal by a rescaling method. As the project goes on, it turns out that this rescaling method, naturally arising from Arakelov geometry, can be also used to prove new results in the geometric setting. For example, we have treated fibered surfaces in \cite{YZ2}. Moreover, by constructing fibrations, such an idea can be used to treat projective varieties of arbitrary dimensions by inductions. These geometric results in turn are the basis of the arithmetic versions in arbitrary dimension. These are the main ideas of the current paper.

\subsection{Geometric case}

Let $X$ be a projective variety of dimension $n$ over a field $k$, and let $L$ be a line bundle on $X$.
The \emph{volume} of $L$ is defined to be
$$
\vol(L): = \limsup_{N \to \infty} \frac{h^0(NL)}{N^n/n!}.
$$
Here we write $NL$ for $L^{\otimes N}$. In fact, we take the convention of writing tensor products of line bundles additively throughout this paper.

It is known that the ``$\limsup$" on the right hand side is actually a limit. See \cite{La} for example. Then we have the following expansion:
$$
h^0(NL) = \frac{1}{n!} \vol(L)N^n + o(N^n), \quad N \to \infty.
$$
The goal of this paper in the geometric case is to provide an effective version of the expansion in the ``$\leq$'' direction.

To introduce the result, we first introduce a basic invariant $\varepsilon(L)$ of $L$.
Recall that a line bundle $M$ on $X$ is \emph{pseudo-effective} if
$$
M\cdot A_1 \cdots A_{n-1} \ge 0
$$
for any nef line bundles $A_1, \cdots, A_{n-1}$ on $X$.
Let $B$ be any big and base-point-free line bundle on $X$. Denote by $\lambda_{L, B}$ the smallest non-negative real number such that
$$
\lambda_{L, B} B- L
$$
is pseudo-effective. We define
$$
\varepsilon(L, B) := (\lambda_{L, B}+1)^{n-1} B^n.
$$
Define
$$
\varepsilon(L) := \inf_B \varepsilon(\CL, B),
$$
where the infimum is taken over all big and base-point-free line bundles $B$ on $X$.
The main result in the geometric case is as follows.

\begin{theorem} \label{main1}
Let $X$ be a geometrically integral projective variety of dimension $n$ over a field $k$. Let $L$ be a line bundle on $X$.
Then
$$
h^0(L) \leq \frac{1}{n!} \vol(L) + n\ \varepsilon(L).
$$
\end{theorem}

\medskip

When $n=1$, the theorem is just the classical $h^0(L)\leq \deg(L)+1$.
When $n=2$, it generalizes the classical Noether inequality on surfaces.
One can also compare it with the result of Shin \cite{Sh}, which is quoted as Theorem \ref{Sh} in our current paper.
We refer to the introductions of \cite{YZ1, YZ2} for more historical accounts.

The theorem is an effective version of the asymptotic expansion of $h^0(N L)$. In fact, it is easy to have
$$
\varepsilon(N L) \le N^{n-1} \varepsilon(L).
$$
Then the result for $NL$ gives
$$
h^0(NL)
\leq
\frac{1}{n!} \vol(N L) + n\ \varepsilon(N L)
 \leq  \frac{1}{n!} \vol(L) N^n+  n\ \varepsilon(L) N^{n-1}.
$$
This gives an effective version of the asymptotic expansion.

If $L$ is big and base-point-free, then $\vol(L)=L^n$ and
$$\varepsilon(N L) \leq  \varepsilon(N L, L)=(N+1)^{n-1} L^n.$$
It follows that the theorem becomes
$$
h^0(NL) \leq \frac{1}{n!}   N^n L^n+ n(N+1)^{n-1} L^n.
$$
This is an effective version of the Hilbert--Samuel formula.
One can compare it with the main result of Koll\'ar--Matsusaka in \cite{KM} and \cite{Ma}.
Under similar assumptions on $L$ and assuming that $k$ has characteristic zero,
their result asserts that
$$\left|h^0(NL) - \frac{1}{n!}   N^n L^n\right| \leq Q_n(L^n, L^{n-1}K_X, N),$$
where $Q_n$ is a (universal) polynomial of three variables whose degree in the third variable is at most $n-1$.
Our result here gives an explicit and simple form of $Q_n$ for the upper bound part, which does not involve $L^{n-1}K_X$, as expected by them.

The result of Koll\'ar--Matsusaka is generalized by Luo \cite{Lu} to the case that $L$ is big and nef, where $Q_n$ is replaced by a polynomial in $N$ of degree at most $n-1$, whose coefficients are determined by $L^{n}$ and $L^{n-1}K_X$.
To compare our result with it, we raise the question whether $\varepsilon(L)$ can be bounded in terms of $L^{n}$ and $L^{n-1}K_X$ if $L$ is big and nef.

On the other hand, it is worth noting that our theorem is true for any line bundle $L$, which does not restrict to multiples of the same line bundle.

The theorem is accurate when $L$ is ``large,'' but it is not so when $L$ is ``small.''
In the latter case, we propose a more delicate bound.

\begin{theorem} \label{main2}
Let $X$ be a smooth and geometrically integral projective variety of dimension $n$ over a field $k$ of characteristic zero. Let $L$ be a line bundle on $X$. Assume that $\omega_X-L$ is pseudo-effective.
Then
$$
h^0(L) \leq \frac{1}{2(n!)} \vol(L) + n\ \varepsilon(L).
$$
\end{theorem}

\medskip

When $n=1$, the theorem is essentially Clifford's theorem
$$h^0(L)\leq \frac12\deg(L)+1$$ for special line bundles.
When $n=2$, it is very close to the main theorem of \cite{YZ2}.
We still refer to loc. cit. for more historical accounts.

The theorem is also proved in the recent work \cite{Zht} by one of our authors, but with a more complicated ``error term.''
As in loc. cit., one can apply the above result to prove Severi's conjecture in high dimensions.

Besides the major assumption that  $\omega_X-L$ is pseudo-effective, there are two extra assumptions in the theorem. First, the assumption of characteristic zero is made to use Hironaka's resolution of singularities. Second, the assumption that $X$ is smooth can be weakened to that $X$ has canonical singularities by applying resolution of singularities.
Because resolution of singularities is known for algebraic 3-folds of positive characteristics (cf. \cite{Ab, Cu, CP1, CP2}), the theorem is true in the case that $\mathrm{char}(k)>0$ and $n=3$.

\subsection{Arithmetic case}

Now we describe our arithmetic versions of the above theorems.

Let $K$ be a number field. Let $\CX$ be \emph{an arithmetic variety of dimension $n$ over $O_K$}, i.e, $\CX$ is an $n$-dimensional normal scheme, projective and flat over $O_K$ such that the generic fiber $\CX_{K}$ is \emph{geometrically connected}.
We assume that $\dim \CX \geq 2$ throughout this paper.

By a \emph{hermitian line bundle} on $\CX$, we mean a pair $\lb=(\CL, \| \cdot \|)$, where $\CL$ is an invertible sheaf on $\CX$, and $\| \cdot \|$ is a continuous metric of the line bundle $\CL(\CC)$ on $\CX (\CC)$, invariant under the complex conjugation.

For any hermitian line bundle $\lb=(\CL, \| \cdot \|)$ on $\CX$, denote the set of \textit{effective sections} as follows:
$$
\Hhat(\lb): = \{s \in H^0(\CX , \CL) : \| s \|_{\sup} \le 1 \}.
$$
Define
$$
\hhat(\lb): = \log \# \Hhat(\lb)
$$
and \emph{the arithmetic volume function}
$$
\volhat(\lb): = \limsup_{N \to \infty} \frac{\hhat(N\lb)}{N^n/n!}.
$$
Recall that a hermitian line bundle $\lb$ is called \emph{big} if $\volhat(\lb)>0$.

Recall that a hermitian line bundle $\lb$ over $\CX$ is called \textit{nef} if it satisfies the following conditions:
\begin{itemize}
\item[(1)] $\deghat(\lb|_Z) \ge 0$ for any integral one-dimensional subscheme $Z$ on $\CX$;
\item[(2)] The metric of $\lb$ is semipositive, i.e., the curvature current of the pull-back $j^*\lb$
via any holomorphic map $j:\Omega\to \CX (\CC)$ from an open complex ball $\Omega$ of dimension $n-1$ is positive.
\end{itemize}
The arithmetic nefness,  introduced by Moriwaki \cite{Mo2}, generalizes the arithmetic ampleness of S. Zhang \cite{Zhs2}. In fact, it is the limit notion of the arithmetic ampleness.

Many results in the geometric case have been proved in the current setting (under substantially more efforts). The following is a list of them (in chronological order) that are most related to the subject of this paper.
\begin{itemize}
\item If $\lb$ is ample, then $\volhat(\lb) = \lb^n$. In other words, one has the arithmetic Hilbert--Samuel formula
$$
\hhat(N\lb) = \frac{N^n}{n!} \lb^n + o(N^n), \quad N \to \infty.
$$
This is essentially due to Gillet-Soul\'e \cite{GS1, GS2} and S. Zhang \cite{Zhs2}.
See \cite[Corollary 2.7]{Yu1} for a brief account.
\item Moriwaki \cite{Mo3} proves the continuity of $\volhat$, and extends the result $\volhat(\lb) = \lb^n$ to any nef line bundle $\lb$.
\item Chen \cite{Ch1} proves that the ``$\limsup$" in the definition of $\volhat$ is
a limit. Thus we have the following expansion:
$$
\hhat(N\lb) = \frac{1}{n!} \volhat(\lb) N^n + o(N^n), \quad N \to \infty.
$$
See Yuan \cite{Yu2} for a proof in terms of Okounkov bodies.
\item Chen \cite{Ch2} and Yuan \cite{Yu2} proves the arithmetic Fujita approximation theorem for big hermitian line bundles.
\item Chen \cite{Ch3} proves the differentiability of the arithmetic volume function, based on the bigness theorem of Yuan \cite{Yu1}, the log-concavity of Yuan \cite{Yu2}, and the arithmetic Fujita approximation theorem above.
\item On arithmetic surfaces, the arithmetic Zariski decomposition is proved by Moriwaki \cite{Mo6}. 
\end{itemize}

To state our main theorem in the arithmetic setting, we need to introduce one more invariant.

Let $\CX$ be an arithmetic variety of dimension $n$ over $O_K$, and $\lb$ be a hermitian line bundle on $\CX$.
Define \emph{the volume derivative}
$$
\dvol(\lb)= \sup_{(\CX ', \ab)}\deg(\mathcal A_K)= \sup_{(\CX ', \ab)}\mathcal A_K^{n-1},
$$
where the supremum takes over every pair $(\CX',\ab)$ consisting of an arithmetic variety $\CX '$ endowed with a birational morphism $\pi:\CX '\to \CX$ and a nef $\QQ$-line bundle $\ab$ on $\CX '$ such that $\pi^*\lb-\ab$ is effective.

Some basic property of the volume derivative is as follows:
\begin{itemize}
\item $\dvol(\lb)>0$ if $\lb$ is big. 
\item If $\lb$ is nef, then $\dvol(\lb)=\mathcal L_K^{n-1}.$
\item If $\dim\CX=2$ and $\lb$ is big, then $\dvol(\lb)=\deg(\mathcal P_K)$, where $\overline{\mathcal P}$ is the positive part of $\lb$ as in the arithmetic Zariski decomposition of Moriwaki \cite{Mo6}. 
\end{itemize}

The definition of $\dvol$ sits in the setting of the arithmetic Fujita approximation of Chen \cite{Ch2} and Yuan \cite{Yu2}. Furthermore, in the sense of Chen \cite{Ch3}, 
the definition is actually a positive intersection number, and thus
$$
\dvol(\lb) =\frac{1}{n[K:\QQ]} \lim_{t\to 0} \frac1t \left(\volhat(\lb(t))-\volhat(\lb)\right),
$$
where $\lb(t)$ denotes the hermitian line bundle obtained by multiplying the hermitian metric of $\lb$ by the constant $e^{-t}$ (at every archimedean place).
This is the reason for the name ``volume derivative.''

One can also interpret $\dvol(\lb)$ as some volume function of some graded linear series on the generic fiber $\CX_K$ encoding certain arithmetic property of $\lb$. See Lemma \ref{limit}. For more properties of $\dvol(\lb)$, we refer to \S \ref{section dvol}.

Finally, we are ready to state our first main theorem in the arithmetic case.

\begin{theorem} \label{main3}
Let $\CX$ be an arithmetic variety of dimension $n$ over $O_K$.
Let $\lb$ be a big hermitian line bundle on $\CX$. Then
$$
\hhat(\lb) \le \left(\frac {1}{n!} + \frac{(n-1) \varepsilon(\CL_K)}{\dvol(\lb)}\right) \volhat(\lb) + 4r \log (3r).
$$
Here $r=h^0(\CL_{\QQ})=[K:\QQ] h^0(\CL_K)$.
\end{theorem}

During the preparation of this paper, similar upper bounds of $h^0(L)$ and  $\hhat(\lb)$ are obtained by Huayi Chen \cite{Ch4} independently. In comparison, the error terms in our Theorem \ref{main1} and Theorem \ref{main3} are more explicit than those in \cite[Theorem 1.2]{Ch4} and \cite[Theorem 1.1]{Ch4}. In comparison with \cite[Theorem 1.3]{Ch4}, our proof also gives
$$
\sum_{i=1}^r \max\{\lambda_i(H^0(\CX,\CL), \|\cdot\|_{\sup}),0\}
\leq  \left(\frac {1}{n!} + \frac{(n-1) \varepsilon(\CL_K)}{\dvol(\lb)}\right) \volhat(\lb). 
$$
To keep this paper as accessible as possible, we do not write our treatment in this setting but leave it to interested readers.

Note that in the theorem, the term $\varepsilon(\CL_K)$, introduced in the geometric case, depends only on the generic fiber $\CL_K$ on $\CX_K$.
Furthermore, if $\lb$ is nef, then the term $\dvol(\lb)$ also depends only on the generic fiber  $\CL_K$ on $\CX_K$.

Let us see the asymptotic of Theorem \ref{main3}.
For any integer $N > 0$, by
$\varepsilon(N \CL_K) \le N^{n-2} \varepsilon(\CL_K)$ and $\dvol(N\lb)=N^{n-1}\dvol(\lb)$,
the theorem gives
$$
\hhat(N\lb) \le \left(\frac {1}{n!} + \frac{(n-1) \varepsilon(\CL_K)}{N\cdot\dvol(\lb)}\right) \volhat(\lb) N^n + 4r_N \log (3r_N).
$$
Here $r_N=h^0(N \CL_\QQ)$ can also be effectively bounded by Theorem \ref{main1}.
Hence, we see the effectivity of the theorem.

If $\lb$ is big and nef with a base-point-free generic fiber $\CL_K$, then as in the geometric case, the theorem gives
$$
\hhat(N\lb) \le \frac {1}{n!}   N^n \lb^n
+ (n-1) N (N+1)^{n-2} \lb^n
+ 4r_N \log (3r_N),
$$
where
$$
r_N/[K:\QQ] \leq \frac{1}{(n-1)!}   N^{n-1} \CL_K^{n-1}+ (n-1)(N+1)^{n-2} \CL_K^{n-1}
$$
by Theorem \ref{main1}.

When $n=2$, if the generic fiber of $\CX$ has positive genus, with a minor assumption,
\cite[Theorem B]{YZ1} actually gives
$$
\hhat(\lb) \le \frac {1}{2}  \volhat(\lb) + 4r \log (3r).
$$
In another word, the ``error term'' disappears here.
However, Theorem \ref{main3} applies to any big line bundles in any dimensions.
It is a philosophy that appropriate assumptions of general type should make the ``error term'' cleaner, but it is really complicated to carry it out for high dimensions. However, in dimension three, see the clean result in Theorem \ref{3fold}.

Similar to Theorem \ref{main2}, we have the following more delicate bound when $\lb$ is ``small'' (on the generic fiber). It is a generalization of \cite[Theorem C]{YZ1}.

\begin{theorem} \label{main4}
Let $\CX$ be an arithmetic variety of dimension $n$ over $O_K$.
Let $\lb$ be a big hermitian line bundle on $\CX$.
Assume that $\CX_K$ is smooth, and $\omega_{\CX_K}-\CL_K$ is pseudo-effective.
Then
$$
\hhat(\lb) \le \left(\frac {1}{2(n!)} + \frac{(n-1) \varepsilon(\CL_K)}{\dvol(\lb)}\right) \volhat(\lb) + 4r \log (3r).
$$
Here $r=h^0(\CL_{\QQ})=[K:\QQ] h^0(\CL_K)$.
\end{theorem}

\begin{remark}
A related result is the upper bound of the arithmetic degree of the push-forward of $\lb$
to $O_K$ by Bost \cite{Bo}. The result is based on the Chow stability method, which can be viewed as an arithmetic version of Hilbert stability of Cornalba--Harris \cite{CH}.
We refer to \cite{YZ1} for some comparisons for arithmetic surfaces.
\end{remark}

At last, we present a clean theorem for arithmetic 3-folds (under more assumptions).

\begin{theorem} \label{3fold}
Let $\CX$ be an arithmetic 3-fold over $O_K$ such that the Kodaira dimension $\kappa(\CX_K)$ of $\CX_K$ is nonnegative and that $\CX_K$ has no elliptic or hyperelliptic pencil. Let $\lb$ be a nef hermitian line bundle on $\CX$ such that the natural rational map $\phi_{\CL_K}: \CX_K\dashrightarrow \mathbb P(H^0(\CL_K))$ is generically finite. Then
$$
\hhat(\lb) \le \left(\frac 16 + \frac{2}{d_K} \right) \lb^3 + 4r \log (3r).
$$
Here $d_K = \deg (\CL_K)$ and $r=h^0(\CL_\QQ)$.
\end{theorem}

\subsection{Ideas of proofs}

The major ideas to prove the theorems are still the rescaling method in \cite{YZ1, YZ2}. However, we do have many innovations to overcome the difficulties in high dimensions.

\subsubsection*{The rescaling method}

For convenience of readers, we first sketch the rescaling method here.
Take Theorem \ref{main1} for example.
Assume that $\lb$ is nef for simplicity here, and we will come back to discuss the extension to the general case later.
We need to give a suitable upper bound of
$$\Delta(\lb)=\hsefhat(\lb)-\frac1{n!} \lb^n.$$
Here $\hsefhat(\lb)$ counts the strictly effective sections,
which is very close to but more convenient than $\hhat(\lb)$.
We first find the largest constant $c\geq 0$ such that
$$
\lb(-c) = (\mathcal L, e^c\| \cdot \|)
$$
is still nef on $\CX$.
It is easy to control $\Delta(\lb)$ by $\Delta(\lb(-c))$.
Then the problem is reduced to $\lb(-c)$.

The key is that $\Hsefhat(\lb(-c))$ is never base-point-free.
By blowing-up the base locus, we obtain a birational morphism $\pi:\CX_1\to \CX$ with a decomposition
$$
\pi^*\lb(-c)= \lb_1+\eb_1.
$$
Here $\eb_1$ is an effective hermitian line bundle associated to the base locus of $\pi^*\lb(-c)$, and $\lb_1$ is base-point-free whose strictly effective sections are bijective to those of $\lb(-c)$.  Then it is easy to control $\Delta(\lb(-c))$ by $\Delta(\lb_1)$.
And the problem is reduced to $\lb_1$.

Keep the reduction process. We obtain $\lb_2, \lb_3, \cdots.$
The key property for the construction is the strict inequality
$$\hsefhat(\lb)>\hsefhat(\lb_1)>\hsefhat(\lb_2) >\cdots.$$
It follows that the process terminates after finitely many steps.
We eventually end up with $\lb_n$ such that $\lb_n(-c_n)$ has no strictly
effective sections. It leads to the proof of the theorem.

\subsubsection*{New ingredients}

The following are some major innovations in this paper.

\begin{itemize}
\item[(1)] The interaction between the geometric case and the arithmetic case. Our proofs of the geometric case are inspired by the rescaling method from the arithmetic case. To apply the method to the geometric case, we construct a fibration of the ambient projective variety over a curve, which mimics the arithmetic setting. To pass from the fibers to the ambient variety, an induction argument is used naturally. On the other hand, the proofs in the arithmetic case use the results in the geometric case.
\item[(2)]
We introduce the invariant $\varepsilon(L)$ to bound the ``error terms.'' It really simplifies the estimates and makes it possible to write down the final inequalities in very general settings in high dimensions.
\item[(3)] In the arithmetic case, the proofs of Theorem \ref{main3} and Theorem \ref{main4} for nef $\lb$, based on Theorem \ref{main1} and Theorem \ref{main2},
are more or less similar to the proofs in \cite{YZ1}. However, the proofs for general $\lb$ are more subtle. In fact, even formulations of the theorems are not obvious. Our new idea is to introduce the derivative volume $\dvol(\lb)$, as a basic invariant of $\lb$.
Then the proof is extended to the general case by the arithmetic Fujita approximation of Chen \cite{Ch2} and Yuan \cite{Yu2} and differentiation theorem of the arithmetic volume function of Chen \cite{Ch3}.
\end{itemize}

\

\noindent\textbf{Acknowledgments.}
The authors would like to thank Huayi Chen and the anonymous referees of our article \cite{YZ1}, who bring insights of the current high-dimensional setting.

The first author is supported by the NSF grant DMS-1330987. The second author
is supported by an NSERC discovery grant.

\section{Geometric case}

The goal of this section is to prove Theorem \ref{main1} and Theorem \ref{main2}.
After introducing the filtration of line bundles, we finish the proofs by collecting the numerical inequalities.

By passing to the algebraic closure, we can assume that $k$ is algebraically closed everywhere in this section.

\subsection{Filtration of line bundles}

This section is a high-dimensional analogue of the construction in \cite{YZ2}.
We include all the details here for completeness.

\subsubsection*{Basic construction}

Let $L$ be any line bundle on a projective variety $X$ with $h^0(L)>0$.
There is a canonical way to separate $L$ from its base locus by blow-up $X$, which is essential in our proof. So we recall it here.

Let $Z$ be the base locus of $L$ in $X$, i.e., the closed subscheme defined by the ideal sheaf given by the image of the composition
$$
H^0(X,L)\times L^\vee \longrightarrow L\times L^\vee \longrightarrow \CO_X .
$$
Note that $Z$ has positive codimension by $h^0(L)>0$.
Let $\pi: X_1\to X$ be the normalization of the blow-up of $X$ along $Z$. Let $Z_1$ be the exceptional divisor on $X_1$, which is the zero locus of the inverse image of the ideal sheaf of
$Z$.
Define a line bundle $L_1=(\pi^*L)\otimes \CO_{X_1}(-Z)$ on $X_1$.
By abuse of notation between line bundles and divisors, we write
$$
\pi^* L = L_1+Z_1.
$$
By definition, the base locus of $\pi^*L$ is $Z_1$.
We have the following properties:
\begin{itemize}
\item There is a canonical isomorphism
$$
H^0(X_1, L_1) \longrightarrow H^0(X_1, \pi^*L) .
$$
\item The line bundle $L_1$ is base-point-free on $X_1$, and thus it is nef.
\item If furthermore $X$ is normal, then the pull-back map gives an isomorphism
$$
H^0(X, L) \longrightarrow H^0(X_1, \pi^*L) .
$$
\end{itemize}
For the last property, the right-hand side is equal to $H^0(X, \pi_*\pi^*L)$,
but $\pi_*\pi^*L=L\otimes \pi_*\CO_{X_1}=L$.

The construction is trivial if $L$ is base-point-free. In the following, we are going to use fibrations to get non-trivial constructions.

\subsubsection*{The filtration}

By a \emph{fibration over a field $k$}, we mean a projective, flat, and geometrically connected morphism $f: X \to C$, where $C$ is a smooth projective curve over $k$, and $X$ is a projective variety over $k$.

Let $f: X \to C$ be a fibration over an algebraically closed field $k$.
For any nef line bundle $L$ on $X$, denote by $e_L$ the positive integer such that
\begin{itemize}
\item $L-e_LF$ is not nef;
\item $L-eF$ is nef for any integer $e<e_L$.
\end{itemize}
Here $F$ denotes a general fiber of $X$ over $C$, and we write $L-eF$ for
the line bundle $L(-eF)$.

Note that $L-e_LF$ is not base-point-free since it is not nef.
So we can perform the basic construction to get a base-point-free line bundle $L_1$, which has ``the same'' global sections as $L$. Keep doing the process on $L_1$, we have the following iterated process.

\begin{theorem}\label{algdecomposition}
Let $f: X \to C$ be a fibration over an algebraically closed field $k$.
Let $L$ be a nef line bundle on $X$ with $h^0(L)>0$.
Then we have a sequence of quadruples
$$
\{(X_i, L_i, Z_i, a_i), \quad i=0, 1, \cdots, N\}
$$
with the following properties:
\begin{itemize}
\item $(X_0, L_0, Z_0, a_0)=(X, L, 0, e_{L_0})$.
\item For any $i=0, \cdots, N-1$, we have $a_i=e_{L_i}$ and $\pi_i: X_{i+1} \to X_i$ is the normalization of the blow-up of $X_i$ along the base locus of $L_i-a_iF_i$, which gives a decomposition
$$
\pi^*_i(L_i-a_iF_i)= L_{i+1} + Z_{i+1}.
$$
Here $Z_{i+1}$ is the exceptional divisor of $\pi_{i}$, the divisor
$F_{i+1}=\pi^*_{i} F_{i}$ and $F_0$ is a general fiber of $X_0$ over $C$.
\item We have
$$h^0(L_0) \geq h^0(L_1)  > h^0(L_2) > \cdots > h^0(L_N)>h^0(L_N-a_NF_N)=0.$$
Here $a_N=e_{L_N}$.
\end{itemize}
\end{theorem}

\begin{proof}
Apply the basic construction repeatedly.
For $i\geq 1$, since $L_i$ is base point free and $e_i>0$, we have
$$
h^0(L_i) > h^0(L_i-e_i F_i)=h^0(L_{i+1}).
$$
Since $h^0$ decreases strictly, the whole process will terminate after finitely many steps.
\end{proof}

\subsubsection*{Numerical inequalities}

Resume the notations in Theorem \ref{algdecomposition}.
Denote $n=\dim X_0$.
Denote
$$
L'_i=L_i-a_iF_i, \quad
r_i=h^0(L_i|_{F_i}), \quad d_i=(L_i|_{F_i})^{n-1}
$$
for $i=0, \cdots, N$.

\begin{prop} \label{algcase1}
For any $j=0, 1, \cdots, N$,
\begin{eqnarray*}
L^n_0 &\ge&  n \sum_{i=0}^j a_i d_i-nd_0, \\
h^0(L_0) &\le& h^0(L'_j) + \sum_{i=0}^j a_{i} r_{i}.
\end{eqnarray*}
\end{prop}

\begin{proof}
Recall that
$$
\pi^*_i L'_i = L_{i+1} + Z_{i+1} = L'_{i+1} + a_{i+1}F_{i+1} + Z_{i+1}.
$$
It implies
$$
\pi^*_i (L'_i+F_i) = (L'_{i+1}+F_{i+1}) + a_{i+1}F_{i+1} + Z_{i+1}.
$$
Note that both $L'_i+F_i$ and $L'_{i+1}+F_{i+1}$ are nef, and $Z_i$ is effective. We have
\begin{eqnarray*}
(L'_i+F_i)^n  \ge  ((L'_{i+1}+F_{i+1}) + a_{i+1}F_{i+1})^n
= (L'_{i+1}+F_{i+1})^n + na_{i+1}d_{i+1}.
\end{eqnarray*}
Summing over $i=0, 1, \cdots, j-1$, we have
$$
(L'_0+F_0)^n  \ge  (L_j'+F_j)^n + n\sum_{i=1}^j a_{i}d_i\ge n\sum_{i=1}^j a_{i}d_i.
$$
It follows that
$$
L^n_0
=(L'_0+F_0)^n+n(a_0-1)d_0
\ge n(a_0-1)d_0 + n \sum_{i=1}^j a_i d_i.
$$
This proves the first inequality.

For the second inequality,  use the exact sequence
$$
0 \longrightarrow H^0(L_{i+1}-F_{i+1}) \longrightarrow H^0(L_{i+1}) \longrightarrow H^0(L_{i+1}|_{F_{i+1}}).
$$
Then
$$
h^0(L_{i+1}-F_{i+1}) \ge h^0(L_{i+1})-h^0(L_{i+1}|_{F_{i+1}})=h^0(L_{i+1})-r_{i+1}.
$$
By induction, we have
$$
h^0(L'_{i+1})=h^0(L_{i+1}-a_{i+1}F_{i+1}) \ge h^0(L_{i+1})-a_{i+1}r_{i+1} = h^0(L'_{i})-a_{i+1}r_{i+1}.
$$
Furthermore,
$$
h^0(L_0) \le h^0(L'_0)+a_0r_0.
$$
The inequality is proved by summing over $i=0, \cdots, j-1$.
\end{proof}

In the following, we give a bound of $L_0^n$ in terms of just $d_0$.

\begin{lemma}\label{algsumai}
With the above notation, we have
$$
L^n_0 \ge d_0\left(\sum_{i=1}^N a_i+na_0 -n\right) \ge d_0 \left(\sum_{i=0}^N a_i - 1\right).
$$
\end{lemma}

\begin{proof}
For $i=0, \cdots, N-1$, denote by
$$
\tau_i=\pi_{i} \circ \cdots \circ \pi_{N-1} : X_N \to X_i
$$
the composition of blow-ups and denote $\tau_N = {\rm id}_{X_N}: X_N \to X_N$.

Write $b=a_1+\cdots+a_N$ and $Z=\tau^*_1Z_1+ \cdots +\tau^*_NZ_N$. We have the following numerical equivalence on $X_N$:
$$
\tau^*_0 L'_0 \sim_{\rm{num}} L'_N + bF_{N} + Z.
$$
Since $L'_0+F_0$ and $L'_N+F_N$ are both nef, it follows that
\begin{eqnarray*}
(L'_0+F_0)^n & = &(\tau^*_0 L'_0 + F_{N})^{n-1}(L'_N+F_N+bF_N+Z) \\
& \ge & (\tau^*_0 L'_0 + F_{N})^{n-1}(L'_N+F_N) + b(\tau^*_0 L'_0 + F_{N})^{n-1} F_N \\
& \ge & b d_0.
\end{eqnarray*}
Combining with
$$
L_0^n-(L'_0+F_0)^n=n(a_0-1)d_0,
$$
the proof of the first inequality is finished. The second one is just because $a_0 \ge 1$.
\end{proof}

\subsection{Proofs of the main theorems}

Here we prove Theorem \ref{main1} and Theorem \ref{main2}.
We are going to prove them by induction on $n=\dim X$.
Note that the case $n=1$ is known. Assume $n\geq 2$ in the following.

\subsubsection*{Nef line bundles and fibrations}

Recall that Theorem \ref{main1} asserts that
$$
h^0(L) \leq \frac{1}{n!} \vol(L) + n\ \varepsilon(L).
$$

First, it is easy to reduce the problem to the case that $L$ is nef.
Assume $h^0(L)>0$, or there were nothing to prove.
By the basic construction, we have a birational morphism $\pi: X_1\to X$ and a decomposition
$$
\pi^*L=L_1+Z_1
$$
with $L_1$ nef and $Z_1$ effective.
Furthermore,
$$
h^0(L)\leq h^0(\pi^* L)=h^0(L_1), \quad
\vol(L) = \vol(L_1+Z_1) \geq \vol(L_1) \geq L_1^n$$
and
$$
\varepsilon(L) \geq \varepsilon(\pi^*L) \geq \varepsilon(L_1).
$$
It follows that the result on $L_1$ implies that on $L$.

Second, it is also easy to reduce it to the case that there is a fibration on $X$.
Let $B$ be any big and base-point-free line bundle on $X$.
We need to prove
$$
h^0(L) \leq \frac{1}{n!} \vol(L) + n\ \varepsilon(L,B).
$$
We will reduce it to the case that there is fibration $f:X \to C$, such that
$$(B-F)\cdot L^{n-1}=(B-F)\cdot B^{n-1}=0, $$
where $F$ denotes a general fiber of $X$ above $C$.

In fact, since $B$ is big and base-point-free, the map $X \to \mathbb P(H^0(X,B))$ is a generically finite morphism. Take two different irreducible elements $W_1,W_2\in |B|$, and let $\sigma: X' \to X$ be the blow-up of $X$ along the intersection $W_1\cdot W_2$.
Denote by $T$ the exceptional divisor.
Then the divisors $\sigma^*W_1-T, \sigma^*W_2-T\in H^0(X',D)$ are disjoint.
Denote $D=\sigma^*B\otimes \CO_{X'}(-T)$, and
denote by $s_1$ and $s_2$ the sections of $H^0(X', D)$ defining
$\sigma^*W_1-T, \sigma^*W_2-T$.
These two sections define a morphism $f:X'\to \mathbb P^1$, which is the fibration.
By construction, $D$ is the linear equivalence class of fibers of $f$, and $\sigma^*B-D=T$ is effective.

Replace $(X, L, B)$ by $(X', \sigma^*L, \sigma^*B)$. Then we have the desired fibration on $X$.

\subsubsection*{General case}

By the above argument, Theorem \ref{main1} is reduced to the following statement:

\emph{
Let $f:X \to C$ be a fibration over an algebraically closed field $k$.
Let $L$ be a nef line bundle on $X$, and $B$ be a big and base-point-free line bundle on $X$.
Denote $n=\dim X$ and denote by $F$ a general fiber of $X$ above $C$.
Assume $(B-F) L^{n-1}=(B-F) B^{n-1}=0$.
Then
$$
h^0(L) \leq \frac{1}{n!} L^n + n\ \varepsilon(L,B).
$$}

Now we prove this statement by induction on $n$.
Apply the construction in Theorem \ref{algdecomposition} to $(X,C, L, F)$.
Still use the same notations as the theorem.
By Proposition \ref{algcase1},
\begin{eqnarray*}
L^n_0 & \ge & n\sum_{i=0}^N a_{i} d_i-nd_0, \\
h^0(L_0) & \le &  \sum_{i=0}^N a_{i} r_{i}.
\end{eqnarray*}
Here we recall $r_i=h^0(L_i|_{F_i})$ and $d_i=(L_i|_{F_i})^{n-1}$.
The difference gives
\begin{eqnarray*}
h^0(L_0) - \frac{1}{n!}L^n_0 & \le &
 \sum_{i=0}^N \left(r_i-\frac{d_i}{(n-1)!}\right) a_i + \frac{d_0}{(n-1)!}.
\end{eqnarray*}

By induction on $n$, we have
$$
r_i\leq \frac{d_i}{(n-1)!} + (n-1)\varepsilon(L_i|_{F_i}).
$$
Therefore,
\begin{eqnarray*}
h^0(L_0) - \frac{1}{n!}L^n_0
& \le &  (n-1) \sum_{i=0}^N a_i\ \varepsilon(L_i|_{F_i})+ \frac{d_0}{(n-1)!} .
\end{eqnarray*}
It suffices to estimate $\varepsilon(L_i|_{F_i})$.

Let $\lambda\geq 0$ be the smallest real number such that $\lambda B-L$ is pseudo-effective.
By definition,
$$
\varepsilon(L,B)=(\lambda +1)^{n-1}B^n.
$$
Denote by $B_i$ the pull-back of $B=B_0$ from $X=X_0$ to $X_i$.
Note that
$\lambda B_i-L_i$ is pseudo-effective, and so is $\lambda (B_i|_{F_i})-L_i|_{F_i}$. Thus
$$
\varepsilon(L_i|_{F_i}) \le
\varepsilon(L_i|_{F_i}, B_i|_{F_i}) \le (\lambda +1)^{n-2} (B_i|_{F_i})^{n-1} = (\lambda +1)^{n-2} B^n
= \frac{\varepsilon(L, B)}{\lambda +1}.
$$

Therefore, we have
\begin{eqnarray*}
h^0(L_0) - \frac{1}{n!}L^n_0
& \le & (n-1)\frac{\varepsilon(L, B)}{\lambda +1} \sum_{i=0}^N a_i + \frac{d_0}{(n-1)!}.
\end{eqnarray*}
It remains to bound $a_0+\cdots+a_N$ and $d_0$.

We first treat the case $L^n>0$.
The pseudo-effectiveness of $\lambda B-L$ has the following consequences:
\begin{itemize}
\item[(1)] It implies $L^n\le \lambda  L^{n-1} B = \lambda d_0$, and thus $d_0>0$.
By Lemma \ref{algsumai},
$$
\sum_{i=0}^N a_i \leq \frac{L^n_0}{d_0} + 1 \leq \lambda +1.
$$
\item[(2)] It implies that
$$
d_0 = L^{n-1}B \le \lambda  L^{n-2} B^2 \le \cdots \le \lambda ^{n-1} B^n.
$$
\end{itemize}
Therefore, if $L^n >0$, we have
\begin{eqnarray*}
h^0(L_0) - \frac{1}{n!}L^n_0
\le  (n-1) \varepsilon(L, B)+ \frac{\lambda ^{n-1} B^n}{(n-1)!}
 \le  n\  \varepsilon(L, B).
\end{eqnarray*}

Next, we treat the case $L^n=0$.
Denote by $W$ an irreducible element of $|B|$.
Consider the exact sequence
$$
0\longrightarrow H^0(L-B) \longrightarrow H^0(L) \longrightarrow H^0(L|_W).
$$
Since $L$ is not big, we have $H^0(L-B)=0$.
It follows that
$$
h^0(L) \le h^0(L|_{W}).
$$
By induction, we have
\begin{eqnarray*}
h^0(L) - \frac{L^n}{n!} & \le & h^0(L|_{W}) \le \frac{L^{n-1}B}{(n-1) !} + (n-1) \varepsilon(L|_{W}, B|_{W}) \\
& \le & \frac{\lambda ^{n-1}B^n}{(n-1)!} + (n-1) (\lambda +1)^{n-2} B^n \le n\ \varepsilon(L, B).
\end{eqnarray*}
This completes the proof.

\subsubsection*{Small case in characteristic zero}

The above proof can be easily modified to prove Theorem \ref{main2}.
Recall that we have extra conditions that $X$ is smooth and $k$ is of characteristic zero, and the crucial assumption that $\omega_X-L$ is pseudo-effective.
We need to strengthen the above result to
$$
h^0(L) \leq \frac{1}{2(n!)} \vol(L) + n\ \varepsilon(L).
$$

Still use the induction method. If $n=1$, then
the assumption that $\omega_X-L$ is pseudo-effective is just $\deg(L)\leq \deg(\omega_X)$.
The inequality becomes
$$
h^0(L) \leq \frac{1}{2} \vol(L) + 1.
$$
This is Clifford's theorem.
It is well known to be true in the case that $L$ is special, i.e., both $h^0(L)>0$ and $h^0(\omega_X-L)>0$. But the case $h^0(\omega_X-L)=0$ can be proved by the Riemann--Roch theorem.

To mimic the above induction method for general $n\geq 2$, we need to keep track of the pseudo-effectivity of $\omega_X-L$ under blow-up's and fibrations.

For any birational morphism $\pi:X'\to X$ with both $X$ and $X'$ smooth, it is a basic result that the ramification divisor $\omega_{X'}-\pi^*\omega_X$ is effective (or zero).
It follows that the pseudo-effectivity of $\omega_X-L$  implies that of
$\omega_{X'}-\pi^*L$. Hence, we can always replace $(X,L)$ by $(X', \pi^*L)$ for any smooth variety $X'$ with a birational morphism to $X$.

With this property, we can still reduce the problem to the case that $L$ is nef.
As before, we can assume that $X$ is endowed with a fibration $f:X \to C$.
Here $X$ is assumed to be smooth by resolution of singularity.

Now we are in the situation to perform the reduction process.
Proposition \ref{algcase1} implies
\begin{eqnarray*}
h^0(L_0) - \frac{1}{2(n!)}L^n_0 & \le &
 \sum_{i=0}^N \left(r_i-\frac{d_i}{2(n-1)!}\right) a_i + \frac{d_0}{2(n-1)!}.
\end{eqnarray*}

Note that the general fiber $F$ of $f$ is also smooth, and the adjunction formula gives
$\omega_{F}=\omega_{X}|_F$. It follows that the pseudo-effectivity of $\omega_X-L$ implies that of $\omega_{F}-L|_F$. Thus the induction assumption applies to the line bundle $L|_F$ on $F$.
The rest of the proof goes through without any extra difficulty.

\section{Arithmetic case}

The goal of this section is to prove Theorem \ref{main3}, Theorem \ref{main4} and Theorem \ref{3fold}. The plan of this section is as follows. In \S 3.1, we recall some basic results in Arakelov geometry. In \S 3.2, we study the basic invariant $\dvol$ and give the new interpretation in terms of the arithmetic Fujita approximation, which will be needed to prove the main theorems. In \S 3.3, we introduce the filtration construction of hermitian line bundles, which are high-dimensional version of that in \cite{YZ1}. In \S 3.4, we prove
Theorem \ref{main3} and Theorem \ref{main4}. In \S 3.5, we prove Theorem \ref{3fold}.

\subsection{Notations and preliminary results}

This section is essentially a reproduction of some part of \cite{YZ1}.
But we would like to list the results here for readers' convenience.

\subsubsection*{Normed modules}

By a normed $\mathbb Z$-module, we mean a pair $\mbar=(M,\|\cdot\|)$ consisting of
a $\mathbb Z$-module $M$ and an $\mathbb R$-norm $\| \cdot \|$ on
$M_{\mathbb R}=M \otimes_{\mathbb Z} \mathbb R$.
We say that $\mbar$ is a normed free $\mathbb Z$-module of finite rank, if
$M$ is a free $\mathbb Z$-module of finite rank.
This is the case which we will restrict to.

Let $\mbar=(M,\|\cdot\|)$ be a normed free $\mathbb Z$-module of finite rank.
Define
$$
\Hhat(\mbar)=\{m \in M: \| m \| \le 1\}, \quad
\Hhat_{\rm sef}(\mbar)=\{m \in M: \| m \| <1\},
$$
and
$$
\hhat(\mbar)= \log \# \Hhat(M), \quad \hhat_{\rm sef}(\mbar)= \log \# \Hhat_{\rm
sef}(M).
$$
The Euler characteristic of $\mbar$ is defined by
$$\chi(\mbar)=\log
\frac{\mathrm{vol}(B(M))}{\mathrm{vol}(M_{\RR}/M)},$$
where
$B(M) = \{ x \in M_\mathbb R: \| x \| \le 1 \}$
is a convex body in $M_{\RR}$.

For any $\alpha \in \mathbb R$, define
$$\mbar(\alpha)=(M, e^{-\alpha} \| \cdot \| ).$$
Since $\hhat_{\rm sef}(\mbar)$ is finite, it is easy to have
$$
\hhat_{\rm sef}(\mbar)= \lim_{\alpha\to 0^{\text{-}}}  \hhat(\mbar(\alpha)).
$$
Then many results on $\hhat$ can be transfered to $\hsefhat$.

\begin{prop} \label{norm1} \cite[Propostion 2.1]{YZ1}
Let $\mbar=(M,\|\cdot\|)$ be a normed free module of rank $r$. The following are
true:
\begin{itemize}
	\item[(1)] For any $\alpha\geq 0$, one has
$$
\begin{array}{rllll}
\hhat(\mbar(-\alpha)) &\leq& \hhat(\mbar) &\leq& \hhat(\mbar(-\alpha))+
r\alpha+r\log3, \\
\hsefhat(\mbar(-\alpha)) &\leq& \hsefhat(\mbar) &\leq& \hsefhat(\mbar(-\alpha))+
r\alpha+r\log3.
\end{array}
$$
	\item[(2)] One has
	$$ \hsefhat(\mbar) \leq \hhat(\mbar) \leq \hsefhat(\mbar)+r\log3. $$
\end{itemize}
\end{prop}

The following filtration version is based on the successive minima of Gillet--Soul\'e \cite{GS1}.

\begin{prop} \label{secred} \cite[Proposition 2.3]{YZ1}
Let $\overline M=(M, \|\cdot\|)$ be a normed free $\mathbb Z$-module of finite
rank.
Let $0=\alpha_0\leq \alpha_1\leq \cdots \leq \alpha_n$ be an increasing
sequence.
For $0\leq i \leq n$, denote by $r_i$ the rank of the submodule of $M$ generated
by $\widehat H^0(\overline M(-\alpha_i))$. Then
\begin{eqnarray*}
\widehat h^0(\overline M) & \leq& \widehat h^0(\overline M(-\alpha_n))+
\sum_{i=1}^{n} r_{i-1} (\alpha_{i}-\alpha_{i-1})  + 4r_0 \log r_0+2r_0 \log 3,\\
\widehat h^0(\overline M) &\geq&
\sum_{i=1}^{n} r_i(\alpha_i- \alpha_{i-1})
- 2r_0 \log r_0-r_0 \log 3 .
\end{eqnarray*}
The same results hold for the pair
$(\widehat h^0_{\rm sef}(\overline M),\widehat h^0_{\rm sef}(\overline
M(-\alpha_n))).$
\end{prop}

\subsubsection*{Effective sections}

Let $\CX$ be an arithmetic variety, and $\lb=(\CL, \| \cdot \|)$ be a hermitian line bundle over $\CX$.
We introduce the following notation.

Recall that the set of \emph{effective sections} is
$$
\Hhat(\CX ,\lb)= \{ s \in H^0(\CX , \mathcal L): \| s \|_{\sup} \le 1\}.
$$
Define the set of \emph{strictly effective sections} to be
$$
\Hhat_{\rm sef}(\CX ,\lb)= \{ s \in H^0(\CX , \mathcal L): \| s \|_{\sup} < 1\}.
$$
Denote
$$
\hhat(\CX ,\lb)= \log\# \Hhat(\CX ,\lb), \quad \hhat_{\rm sef}(\CX ,\lb)= \log\#
\Hhat_{\rm sef}(\CX ,\lb).
$$
We say that $\lb$ is \emph{effective} (resp. \emph{strictly effective}) if
$\hhat(\CX ,\lb)\neq 0$ (resp. $\hhat_{\rm sef}(\CX ,\lb)\neq0$).

We usually omit $\CX$ in the above notations. For example, $\Hhat(\CX ,\lb)$ is
written as $\Hhat(\lb)$.

Note that $\overline M=(H^0(\CX , \mathcal L), \| \cdot \|_{\sup})$ is a normed
$\ZZ$-module.
The definitions are compatible in that
$$
\Hhat(\lb),\quad \Hsefhat(\lb), \quad \hhat(\lb), \quad \hsefhat(\lb)
$$
are identical to
$$
\Hhat(\mbar),\quad \Hsefhat(\mbar), \quad \hhat(\mbar), \quad \hsefhat(\mbar).
$$
Hence, the results in last section can be applied here.

For example, Proposition \ref{norm1} gives
$$
\hhat_{\rm sef}(\lb)\leq \hhat(\lb) \leq \hhat_{\rm sef}(\lb)
+ h^0(\CL_\QQ)\log 3.
$$

Note that if $\CX$ is also defined over $\Spec \, O_K$ for some number field $K$, then we obtain two projective varieties $\CX_\QQ=\CX \times_{\ZZ}\QQ$ and
$\CX_K=\CX \times_{O_K}K$,
and two line bundles $\CL_\QQ$ and $\CL_K$.
It is easy to have
$$
h^0(\CL_\QQ)=[K:\QQ]h^0(\CL_K).
$$
Moreover, we can define the degree of $\CL_\QQ$ on $\CX_\QQ$  to be $d_\QQ=\deg(\CL_\QQ)=\CL^{n-1}_\QQ$.
Similarly, we have
$$
d_\QQ = [K : \QQ] d_K.
$$

\subsubsection*{Change of metrics}

For any continuous function $f: \CX (\mathbb C) \rightarrow \mathbb R$, denote
$$
\lb(f)=(\mathcal{L}, e^{-f}\| \cdot \|).
$$
In particular,
$\ob(f)=(\mathcal{O}_\CX , e^{-f})$ is the trivial line bundle with the metric
sending the section 1 to $e^{-f}$.
The case $\ob_\CX =\ob(0)$ is exactly the trivial hermitian line bundle on $\CX$.

If $c>0$ is a constant, one has
$$
\hhat(\lb(-c))\leq \hhat(\lb) \leq \hhat(\lb(-c))+h^0(\CL_\QQ)(c+\log 3),
$$
$$
\hhat_{\rm sef}(\lb(-c))\leq \hhat(\lb) \leq \hhat_{\rm
sef}(\lb(-c))+h^0(\CL_\QQ)(c+\log 3).
$$
These also follow from Proposition \ref{norm1}.

\subsubsection*{Base loci}

Let $H$ denote $\Hhat(\lb)$ or $\Hhat_{\rm sef}(\lb)$ in the following.
Consider the natural map
$$
H \times \CL^\vee \longrightarrow \CL \times \CL^\vee \longrightarrow
\mathcal{O}_\CX .
$$
The image of the composition generates an ideal sheaf of $\CO_\CX$.
The zero locus of this ideal sheaf, defined as a closed subscheme of $\CX$, is
called \emph{the base locus of} $H$
in $\CX$. The union of the irreducible components of codimension one of the base
locus is called \emph{the fixed part of} $H$ in $\CX$.

\subsubsection*{Absolute minima}

For any irreducible horizontal arithmetic curve $D$ of $\CX$, define the normalized height
$$
h_{\lb}(D)=\frac{\widehat\deg(\lb|_ D)}{\deg D_{\mathbb Q}}.
$$
Define the \emph{absolute minimum} $e_{\lb}$ of $\lb$ to be
$$
e_{\lb}=\inf_{D} h_{\lb}(D).
$$
It is easy to verify that
$$
e_{\lb(\alpha)}=e_{\lb}+\alpha, \quad \alpha\in\RR.
$$

If $\lb$ is nef, the absolute minimum $e_{\lb}\geq 0$, and
 $\lb(-e_{\lb})$ is a nef line bundle whose absolute minimum is zero.
It is a very important fact in our treatment in the following.

We refer to \cite{Zhs1, Zhs2} for more results on the minima of $\lb$ for nef
hermitian line bundles.

\subsection{Volume derivative} \label{section dvol}

Let $\lb$ be a hermitian line bundle on an arithmetic variety $\CX$ of dimension $n$ over $O_K$. 
Recall that the volume derivative is defined by 
$$
\dvol(\lb)= \sup_{(\CX ', \ab)}\mathcal A_K^{n-1},
$$
where $\CX '$ is any arithmetic variety endowed with a birational morphism $\pi:\CX '\to \CX$, and $\ab$ is any nef $\QQ$-line bundle on $\CX '$ such that $\pi^*\lb-\ab$ is effective.
The goal here is to give more interpretations of this basic invariant.

\subsubsection*{Derivative of the arithmetic volume function}

The first result is the following interpretation. 

\begin{lemma} \label{derivative}
For any big hermitian line bundle $\lb$,
$$
\dvol(\lb)=\frac{1}{n[K:\QQ]} \lim_{t\to 0} \frac1t \left(\volhat(\lb(t))-\volhat(\lb)\right).
$$
\end{lemma}

The lemma is an example of the differentiation of the arithmetic volume function of Chen \cite{Ch3}. We start with some general notations to introduce the result. 

Denote by $\Pichat(\CX )$ the group of hermitian line bundles on $\CX$, and by $\Pichat(\CX )_\QQ=\Pichat(\CX )\otimes_{\ZZ}\QQ$ the group of hermitian $\QQ$-line bundles on $\CX$. Denote by 
$$
\PPichat(\CX )_\QQ= \varinjlim_{\CX '}\Pichat(\CX ')_\QQ,
$$
where the direct limit is taken over all arithmetic varieties $\CX '$ with a birational morphisms $\CX '\to \CX$, and the transition maps between different $\Pichat(\CX ')$ as just the pull-back of line bundles. 

An element of $\PPichat(\CX )_\QQ$ is said to be effective (resp. big or nef) if some positive multiple of it can be represented by an effective (resp. big or nef) hermitian line bundle on $\CX '$. Denote by $\NNefhat(\CX )_\QQ$, $\BBighat(\CX )_\QQ$ and $\EEffhat(\CX )_\QQ$ respectively the cone of nef, big, and effective elements of $\PPichat(\CX )_\QQ$. 

For two elements $\lb_1$ and $\lb_2$, we say that $\lb_1$ dominates $\lb_2$ if $\lb_1-\lb_2$ is effective. In that case, we write $\lb_1\succ\lb_2$ or $\lb_2\prec\lb_1$.

The volume function extends to $\PPichat(\CX )_\QQ$ by homogeneity, and the intersection pairing extends to $\PPichat(\CX )_\QQ$ naturally. 
In particular, $\volhat(\ab)=\ab^n$ for any $\ab\in \NNefhat(\CX )_\QQ$. 

Let $\lb$ be an element of $\BBighat(\CX )_\QQ$. 
The arithmetic Fujita approximation of Chen \cite{Ch2} and Yuan \cite{Yu1} asserts that
$$
\volhat(\lb)=\sup_{\ab\in \NNefhat(\CX )_\QQ, \ \ab\prec \lb} \volhat(\ab) 
$$

The main result of Chen \cite{Ch3} is the following theorem is as follows. 
For any $\lb\in\BBighat(\CX )_\QQ$ and $\overline{\mathcal M}\in \PPichat(\CX )_\QQ$,
the derivative 
$$
\lim_{t\to 0} \frac1t \left(\volhat(\lb+t\mb)-\volhat(\lb)\right)= 
n\ \langle\lb^{n-1}\rangle\cdot \overline{\mathcal M}.
$$
Here \emph{the positive intersection number} for $\overline{\mathcal M}\in \EEffhat(\CX )_\QQ$ is defined by 
$$
\langle\lb^{n-1}\rangle\cdot \overline{\mathcal M}
:=\sup_{\ab\in \NNefhat(\CX )_\QQ, \ \ab\prec \lb} 
\ab^{n-1} \cdot \overline{\mathcal M}.
$$
It turns out that the positive intersection number is additive in $\overline{\mathcal M}$, and thus extends to any 
$\overline{\mathcal M}\in \PPichat(\CX )_\QQ$ by linearity. 

Go back to the volume derivative. 
Take $\mb=\ob(1)$. 
We immediately have
$$
\lim_{t\to 0} \frac1t \left(\volhat(\lb(t))-\volhat(\lb)\right)= 
n\ \langle\lb^{n-1}\rangle\cdot \ob(1),
$$
where
$$
  \langle\lb^{n-1}\rangle\cdot \ob(1)
= [K:\QQ] \sup_{\ab\in \NNefhat(\CX )_\QQ, \ \ab\prec \lb} 
\mathcal A_K^{n-1}.
$$
This proves the lemma.

\subsubsection*{Interpretation by algebraic linear series}

We can also interpret $\dvol$ as the volume function of certain graded linear series on 
the generic fiber $\CX_K$.
In the following, denote by $\langle  \Hhat(\lb)\rangle_K$ the $K$-linear subspace of $H^0(\CL_K)$ generated by $\Hhat(\lb)$. 

\begin{prop} \label{limit}
For any big hermitian line bundle $\lb$, 
$$
\dvol(\lb)= \lim_{N \to \infty} \frac{\dim_K  \langle  \Hhat(N\lb)\rangle_K}{N^{n-1}/(n-1)!}.
$$
\end{prop}

This essentially follows from the construction of the arithmetic Fujita approximation by Chen \cite{Ch2}. Here we give another interpretation in the terminology of Boucksom--Chen \cite{BC}, since it contains more information.

We first introduce some notations. 
For any $t\in \RR$, denote 
$$
R(t)= \bigoplus_{N= 0}^\infty \langle  \Hhat( N\lb_{-t} )\rangle_K.
$$
Here we write $\lb_{-t}$ for $\lb(-t)$ to avoid the confusion by the multiplication by $N$. 
Then $R(t)$ is a graded subring of the section ring
$$
R(\CL_K)= \bigoplus_{N= 0}^\infty  H^0( N\CL_K ).
$$ 

Fix an algebraic point of $X$ and a local coordinate at this point. 
By the construction in Lazarsfeld--Musta\c t\v a \cite{LM}, we obtain the Okounkov body 
$\Delta(\CL_K)$ of $\CL_K$, which is a convex body in $\RR^{n-1}$. 
Furthermore, we also have an Okounkov body $\Delta(t)$ for each graded ring $R(t)$. 
Note that $\Delta(t)\subset \Delta(\CL_K)$ by definition. 

From the construction of $\Delta(t)$, we can see that 
$$
\vol(\Delta(0))=\lim_{N \to \infty} \frac{\dim_K  \langle  \Hhat( N\lb )\rangle_K}{N^{n-1}}.
$$
A property hidden in the equality is that the right-hand side converges.

As in \cite{BC}, define a function
$$
G_{\lb}: \Delta(\CL_K) \longrightarrow \RR
$$
by 
$$
G_{\lb}(x):= \sup \{t\in \RR: x\in \Delta(t)\}.  
$$
Then the main result of \cite{BC} gives
$$
\frac{1}{n![K:\QQ]}\volhat(\lb)=\int_{\Delta(\CL_K)} \max\{G_{\lb}(x),0\}\ dx. 
$$

By definition, 
$$\Delta(0)=\{  x\in \Delta(\CL_K): G_{\lb}(x)\geq 0 \}.$$
It follows that we can just write 
$$
\frac{1}{n![K:\QQ]}\volhat(\lb)=\int_{\Delta(0)}  G_{\lb}(x) \ dx. 
$$
By definition, $G_{\lb(t)}(x)=G_{\lb}(x)+t$. 
It follows that 
\begin{eqnarray*}
\frac{1}{n![K:\QQ]}\volhat(\lb(t))
&=&\int_{\Delta(-t)} (G_{\lb}(x)+t) dx\\
&=&\int_{\Delta} G_{\lb}(x) dx + \vol(\Delta(0)) t+o(t) . 
\end{eqnarray*}
Hence, 
$$
\frac{1}{n![K:\QQ]} \lim_{t\to 0} \frac1t \left(\volhat(\lb(t))-\volhat(\lb)\right)
=\vol(\Delta(0)), 
$$
which is essentially the equality in the proposition.

\subsection{Filtration of hermitian line bundles}

This section is the high-dimensional analogue of the construction in \cite{YZ1}.
One difference is that we need blow-ups to finish the decomposition in high-dimensional case.

\subsubsection*{The basic construction}

Our goal of this section is to introduce a basic decomposition of hermitian line
bundles on arithmetic varieties, as high-dimensional generalizations of
\cite[Theorem 3.2]{YZ1}. It is a decomposition keeping $\Hhat_{\rm sef}(\lb)$.

\begin{theorem}\label{degred1}

Let $\CX$ be a normal arithmetic variety,
and $\lb$ be a hermitian line bundle with $\hhat_{\rm sef}(\lb)\neq 0$.
Then there exist a birational morphism $\pi: \CX_1 \to \CX$ and a decomposition
$$
\pi^*\lb=\lb_1+\eb_1
$$
where $\eb_1$ is an effective hermitian line bundle on $\CX$, and $\lb_1$ is a nef hermitian line bundle on
$\CX_1$ satisfying the following conditions:
\begin{itemize}
\item There is an effective section $e\in \Hhat(\eb)$ such that
$\mathrm{div}(e)$ is the base locus of $\Hhat_{\rm sef}(\pi^*\lb)$ in $\CX_1$.
\item The map $\CL_1\to \pi^*\CL$ defined by tensoring with $e$ induces a bijection
$$\Hhat_{\rm sef}(\lb_1)  \stackrel{\otimes e}{\longrightarrow}\Hhat_{\rm
sef}(\pi^*\lb).$$
Furthermore, the bijection keeps the supremum norms, i.e.,
$$
\|s\|_{\sup} = \|e\otimes s\|_{\sup}, \quad \forall \ s\in \Hhat_{\rm
sef}(\lb_1).
$$
\end{itemize}
\end{theorem}

The above result is a generalization of \cite[Theorem 3.2]{YZ1} for the arithmetic surface case. One can also obtain a similar decomposition keeping $\Hhat(\lb)$, as a generalization of \cite[Theorem 3.1]{YZ1} .

The proof of the theorem is very similar to that of \cite[Theorem 3.2]{YZ1}, except that we need to blow-up the base loci in the high-dimensional case to make them Cartier divisors.
In the following, we sketch it briefly.

Denote by $Z$ the base locus of $\Hsefhat(\CX , \lb)$ on $\CX$.
Let $\pi: \CX_1 \to \CX$ be the normalization of the blow-up of $\CX$ with center $Z$.
Denote by $\CE_1$ the line bundle on $\CX_1$ associated to $Z$, and by $e \in
H^0(\CX_1, \CE_1)$ the section defining $Z$. Define the line bundle $\CL_1$ on $\CX_1$ by the
decomposition
$$
\pi^*\CL=\CL_1+\CE_1.
$$

Define the metric $\|\cdot\|_{\CE}$ of $\CE$ at $x \in \CX_1(\CC)$ by
$$
\|e(x)\|_{\CE}=\max \{ \|s(x)\|/\| s\|_{\sup}: s\in \Hhat_{\rm sef}(\CX_1,\pi^*\lb), \ s\neq 0
\}.
$$
It is easy to see that $\|e\|_{\CE, \sup}= 1$.
Define the metric $\|\cdot\|_{\CL_1}$ on $\CL_1$ by the decomposition
$$
\pi^*\lb=(\CE,\|\cdot\|_{\CE})+(\CL_1,\|\cdot\|_{\CL_1}).
$$
Set $\eb=(\CE,\|\cdot\|_{\CE})$ and $\lb_1=(\CL_1,\|\cdot\|_{\CL_1})$.
Then the decomposition $\pi^*\lb=\eb+\lb_1$ satisfies the theorem.
The proof is similar to that of \cite[Theorem 3.2]{YZ1}, and we omit it here.

\subsubsection*{The filtration}

In this section, the plan is to write down a filtration of hermitian line bundles by performing the above decomposition repeatedly.

Let $\lb$ be a nef hermitian line bundle.
We are going to apply Theorem \ref{degred1} to reduce $\lb$ to ``smaller'' nef
line bundles.
The problem is that the fixed part of $\lb$ may be empty, and then
Theorem \ref{degred1} is a trivial decomposition.
The idea is to enlarge the metric of $\lb$ by constant multiples to create base
points.
To keep the nefness, the largest constant multiple we can use gives the case
that the absolute minimum is 0.
The following proposition says that the situation exactly meets our requirement.

\begin{prop}\label{fixedpart}

Let $\CX$ be an arithmetic variety,
and $\lb$ be a nef hermitian line bundle on
$\CX$ satisfying
$$\hhat_{\rm sef}(\lb)>0, \quad e_{\lb}=0.$$
Then the base locus of $\Hhat_{\rm sef}(\lb)$
contains a non-empty horizontal part.
\end{prop}

\begin{proof}
We prove this by contradiction. Suppose the base locus of $\Hhat_{\rm sef}(\lb)$ is empty or vertical. Then for any horizontal arithmetic curve $D$ on $\CX$, we can find a nonzero section $s \in \Hsefhat(\lb)$ such that
$s$ does not vanish on $D$. Thus one has
$$
h_{\lb}(D)=\frac{1}{\deg(D_\QQ)} ({\rm div}(s)\cdot D-\log\|s\|(D(\CC))) \geq
-\log\|s\|_{\sup}.
$$
Therefore,
$$e_{\lb} \geq \min_{s\in  \Hsefhat(\lb)-\{0\}} (-\log\|s\|_{\sup})>0.$$
It contradicts to our assumption.
\end{proof}

From the above proposition,  we have the following total construction.

\begin{theorem}\label{degred}

Let $\CX_0$ be a regular arithmetic variety, and $\lb_0$ a nef hermitian line
bundle on $\CX_0$. There exist an integer $N \geq 0$, and a sequence of quadruples
$$
\{(\CX_i, \lb_i, \eb_i, c_i): \ i=0,1, \cdots, N\}
$$
where  $\CX_i$ is a normal arithmetic variety, and $\lb_i$ and $\eb_i$ are hermitian line bundles on $\CX_i$ satisfying the following properties:
\begin{itemize}
\item $(\CX_0, \lb_0, \eb_0, c_0)=(\CX_0, \lb_0, \overline\CO_\CX , e_{\lb_0})$.
\item For any $i=0,\cdots, N$, the constant $c_i=e_{\lb_{i}}\geq 0$ is the absolute minimum of $\lb_{i}$.
\item For any $i=0, \cdots, N-1$, $\pi_i: \CX_{i+1} \to \CX_{i}$ is a birational morphism and
    $$
    \pi^*_i \lb_{i}(-c_{i})=  \lb_{i+1}+ \eb_{i+1}
    $$
    is a decomposition of $\pi^*_i\lb_{i}(-c_i)$ as in Theorem \ref{degred1}.
\item $\hhat_{\rm sef}(\CX , \lb_{i}(-c_i))>0$ for any $i=0, \cdots, N-1$.
\item $\hsefhat(\lb_N(-c_N))=0$.
\end{itemize}
The following are some properties by the construction:
\begin{itemize}
\item For any $i=0,\cdots, N$, $\lb_i$ is nef and every $\eb_i$ is effective.
\item $\hsefhat(\lb_0) \ge \hsefhat(\lb_1) > \hsefhat(\lb_2) > \cdots > \hsefhat(\lb_N)
> \hsefhat(\lb_N(-c_N))=0$.
\item For any $i=0,\cdots, N-1$, there is a section $e_{i+1}\in
\Hhat(\eb_{i+1})$ inducing a bijection
$$
\Hsefhat (\lb_{i+1})  \longrightarrow  \Hsefhat(\pi^*_i\lb_{i}(-c_{i}))
$$
which keeps the supremum norms.
\end{itemize}

\end{theorem}

\begin{proof}
The quadruple $(\CX_{i+1}, \lb_{i+1}, \eb_{i+1}, c_{i+1})$ is obtained by decomposing $\pi^*\lb_{i}(-c_{i})$.
From our construction in Theorem \ref{degred1}, one can see that for $i=1, \cdots, N$,
$\Hsefhat(\lb_i)$ has no base locus but $\Hsefhat(\lb_i(-c_i))$ has. It implies that
$$
\hsefhat(\lb_i) > \hsefhat(\lb_i(-c_i)) = \hsefhat(\lb_{i+1}).
$$
The process terminates if $\hhat_{\rm sef}(\CX , \lb_{i}(-c_i))=0.$
It always terminates since $\hsefhat(\lb_0)$ is finite.
\end{proof}

\subsubsection*{Numerical inequalities}
Recall that Theorem \ref{degred} starts with a nef line bundle $\lb_0$
and constructs the sequence
$$
(\CX_i, \lb_i, \eb_i,c_i), \quad i=0, \cdots, N.
$$
Here $\lb_i$ is nef and $\eb_i$ is effective, and $c_i=e_{\lb_i}\geq 0$.
In particular, $\lb_{i}(-c_{i})$ is still nef.
For any $i=0,\cdots, N-1$, the decomposition
$$
\pi^*_i \lb_{i}(-c_i)=  \lb_{i+1}+ \eb_{i+1}
$$
gives a bijection
$$
\Hhat_{\rm sef}(\lb_{i+1})  \longrightarrow  \Hhat_{\rm sef}(\pi^*_i \lb_{i}(-c_{i})),
$$
which is given by tensoring some distinguished element $e_{i+1} \in \Hhat (\eb_{i+1})$.
It is very important that the bijection keeps the supremum norms.
In the following, we denote
$$
\lb_{i}'=\lb_{i}(-c_i), \quad i=0,\cdots, N.
$$
We also denote $d_{i}=\deg(\CL_{i, \mathbb Q})=\CL^{n-1}_{i, \QQ}$ and $r_i=h^0(\CL_{i,\QQ})$.

\begin{prop} \label{onestep}
For any $j=0,\cdots, N$, one has
\begin{eqnarray*}
\lb_0^n & \geq &  \lb_j'^n + n \sum_{i=0}^{j} d_i c_i,  \\
\hsefhat(\lb_0) &\leq& \hsefhat(\lb_j')+
\sum_{i=0}^{j} r_{i} c_i +4r_0 \log r_0+ 2r_0 \log 3.
\end{eqnarray*}

\end{prop}

\begin{proof}
The results can be proved by the method of \cite[Proposition 4.5]{YZ1}.
We only write a proof for the first inequality here.

By construction, we have
$$
\pi^*_i\lb_{i}'=\lb_{i+1}+ \eb_{i+1}=\lb_{i+1}'+ \eb_{i+1}+\ob(c_{i+1}).$$
Here $\lb_{i}'$ and $\lb_{i+1}'$ are nef, and $\eb_{i+1}$ is effective.
It follows that
\begin{eqnarray*}
\lb_{i}'^n-\lb_{i+1}'^n & = & (\eb_{i+1} +\ob(c_{i+1})) \cdot \left(\sum_{k=0}^{n-1} (\pi^*_i \lb'_i)^k \cdot \lb^{n-1-k}_{i+1} \right)\\
& \geq & \left(\sum_{k=0}^{n-1} (\pi^*_i \lb'_i)^k \cdot \lb^{n-1-k}_{i+1} \right) \cdot  \ob(c_{i+1})\\
& \ge & n d_{i+1} c_{i+1}.
\end{eqnarray*}
Summing over $i=0, \cdots, j-1$, we can get
$$
\lb_0'^n  \geq   \lb_j'^n + n \sum_{i=1}^{j} d_i c_i.
$$
Then the conclusion follows from
$$
\lb_0^n= \lb_0'^n - n d_0 c_0.
$$
\end{proof}

Similar to Lemma \ref{algsumai}, we still have the following result.

\begin{lemma} \label{sumci}
In the setting of Theorem \ref{degred}, we have
$$
\lb_0^n\geq d_0\left(n c_0+\sum_{i=1}^N c_i\right) .
$$
\end{lemma}

\begin{proof}
Let $\fb_i$ be the pull back of $\eb_i$ from $\CX_i$ to $\CX_N$. We
denote
$$\beta=c_1+\cdots+c_N, \quad
\fb=\fb_1+\cdots+\fb_N, \quad
\sigma=\pi_0 \circ \cdots \circ \pi_{N-1}: \CX_N \to \CX_0.$$
Hence we have the decomposition
$$
\sigma^*\lb_0'(-\beta)=\lb_N'+\fb.
$$
Note that $\lb_0'(-\beta)$ is not nef any more.
But we can still have a weaker bound as follows:
\begin{eqnarray*}
\lb_0'^n
&=& (\sigma^*\lb_0')^{n-1}\cdot(\lb_{N}' + \fb + \ob(\beta)) \\
& \geq & (\sigma^*\lb_0')^{n-1} \cdot\lb_{N}'+ d_0 \beta \\
% &=& (\lb_{n}' + \fb+\ob(\beta))\cdot\lb_{n}'+d_0 \beta \\
&\geq& \lb_{N}'^n + d_0 \beta.
\end{eqnarray*}

Combine with
$$\lb^n=\lb^n_0= \lb_0(-c_0)^n+ nd_0c_0. $$
We have
\begin{equation*}
\lb^n \geq \lb_N'^n + d_0 \beta + nd_0c_0 \geq d_0 (nc_0+\beta).
\end{equation*}
The result follows.
\end{proof}

\subsection{Proofs of the main theorems}

In this section, we prove Theorem \ref{main3} and \ref{main4}.

\subsubsection*{Nef case}

Here we start to prove Theorem \ref{main3}.
To illustrate the idea, we first treat the nef case.

Assume that $\lb$ is nef (and big).
Recall that Theorem \ref{main3} asserts
$$
\hhat(\lb) \le \left(\frac {1}{n!} + \frac{(n-1) \varepsilon(\CL_K)}{d/[K:\QQ]}\right) \lb^n +
4r \log (3r).
$$
Here $d=\CL_\QQ^{n-1}$ and $r=h^0(\CL_{\QQ})$.

Apply the construction of Theorem \ref{degred} to $(\CX_0,\lb_0)=(\CX ,\lb)$.
Resume the notations of the theorem.
By Proposition \ref{onestep}, we have
\begin{eqnarray*}
\lb^n & \geq &   n \sum_{i=0}^{N} d_i c_i  , \\
\hsefhat(\lb) & \leq & \sum_{i=0}^{N} r_{i} c_i +4r_0 \log r_0+ 2r_0 \log 3.
\end{eqnarray*}
It follows that
$$
\hsefhat(\lb) - \frac{\lb^n}{n!} \le \sum_{i=0}^N \left(r_i-\frac{d_i}{(n-1)!}\right) c_i + 4r_0\log{r_0}+2r_0\log 3.
$$

The key is to apply Theorem \ref{main1}, the effective bound in the geometric case.
For any $i=0, \cdots, N$,
$$
r_i-\frac{d_i}{(n-1)!} \le (n-1) \varepsilon(\CL_K)  [K:\QQ] .
$$
It follows that
\begin{eqnarray*}
\hsefhat(\lb) - \frac{\lb^n}{n!} & \le &  (n-1) \varepsilon(\CL_K)  [K:\QQ]
\sum_{i=0}^N c_i + 4r_0\log{r_0}+2r_0\log 3.
\end{eqnarray*}

To bound $c_0+\cdots+c_N$, by Lemma \ref{sumci}, we get
$$
\sum_{i=0}^N c_i \leq \frac{1}{d_0} \lb^n.
$$
It follows that
\begin{eqnarray*}
\hsefhat(\lb) - \frac{\lb^n}{n!}
& \le & (n-1)  \varepsilon(\CL_K)  [K:\QQ] \frac {\lb^n}{d_0} + 4r_0\log{r_0} + 2r_0\log 3.
\end{eqnarray*}
Finally, by Proposition \ref{norm1},
$$
\hhat(\lb) \le \hsefhat(\lb) + r_0 \log 3.
$$
This finishes the proof.

We remark that the denominator $d_0>0$ in the current setting.
In fact, $\CL_K$ is nef and big following the assumption that $\lb$ is nef and big.
The nef part is trivial, and the big part is a result of Yuan \cite{Yu1}.

\subsubsection*{General case}

Here prove Theorem \ref{main3} in the full case (that the line bundle is big).
The major difficulty to carry the above proof is to seek a good formulation of Lemma \ref{sumci}. Our idea is to use the arithmetic Fujita approximation to overcome the difficulty.

Recall that the theorem asserts that, for any big hermitian line bundle $\mb$ on $\CX$,
$$
\hhat(\mb) \le \left(\frac {1}{n!} + \frac{(n-1) \varepsilon(\CM_K)}{\dvol(\mb)}\right) \volhat(\mb) + 4s \log (3s).
$$
Here $s=h^0(\CM_{\QQ})=[K:\QQ] h^0(\CM_K)$.

Here we deliberately switch the notation for the line bundle in consideration from $\lb$ to $\mb$, in order to accommodate the notations in Theorem \ref{degred} and afterwards.

Assume that $\hhat(\mb)>0$.
Note that $\mb$ is not necessarily nef, so our first step is to use the key decomposition to make it nef as in the geometric case.
Applying Theorem \ref{degred1} to $\mb$, we have a decomposition
$$
\pi^*\mb=\lb_0+\eb
$$
based on a birational morphism $\pi: \CX_0 \to \CX$.
Here $\eb$ is effective, $\lb_0$ is nef, and $\hhat(\mb)=\hhat(\lb_0)$.
Note the change of notations again.

Next, apply Theorem \ref{degred} to the nef bundle $\lb_0$ over $\CX_0$.
As in the theorem, we get a sequence of quadruples
$$
\{(\CX_i, \lb_i, \eb_i, c_i): \ i=0,1, \cdots, N\}.
$$
By the above argument, we still have
\begin{eqnarray*}
\hsefhat(\lb_0) - \frac{\lb_0^n}{n!} & \le &  (n-1) \varepsilon(\CL_K)  [K:\QQ]
\sum_{i=0}^N c_i + 4r_0\log{r_0}+2r_0\log 3.
\end{eqnarray*}
It is easy to see that it implies
\begin{eqnarray*}
\hsefhat(\mb) - \frac{1}{n!} \volhat(\mb) & \le &  (n-1) \varepsilon(\CM_K)  [K:\QQ]
\sum_{i=0}^N c_i + 4s\log{s}+2s\log 3.
\end{eqnarray*}

It suffices to prove
$$
\sum_{i=0}^N c_i \leq \frac{1}{[K:\QQ]\dvol(\mb)} \volhat(\mb).
$$
Note that Lemma \ref{sumci} gives
$$
\sum_{i=0}^N c_i \leq \frac{1}{d_0} \lb_0^n,
$$
which is not strong enough.
However, the result is actually true for ``bigger'' nef line bundles, and the limit will give what we need.

Resume the notations in \S \ref{section dvol}.
Let $\lb_{-1} \in \NNefhat(\CX )$ be an element such that
$$
\mb\succ \lb_{-1} \succ \lb_0.
$$
Set $c_{-1}=0$.
Add $(\lb_{-1}, c_{-1})$ to the beginning of the sequence $\{(\lb_{i}, c_{i})\}_i$.
It is easy to see that Lemma \ref{sumci} can be applied to the sequence
$$
(\lb_{-1}, c_{-1}), \ (\lb_{0}, c_{0}), \ (\lb_{1}, c_{1}), \cdots, (\lb_{N}, c_{N}).
$$
It is crucial that the lemma only involves intersection numbers (without $\hhat$).
Hence, we have
$$
\sum_{i=0}^N c_i=\sum_{i=-1}^N c_i \leq \frac{1}{(\lb_{-1,\QQ})^{n-1}} (\lb_{-1})^n.
$$
Note that elements $\lb_{-1}$ of $\NNefhat(\CX )$ satisfying $\mb\succ \lb_{-1} \succ \lb_0$ exist by \cite[Proposition 3.3]{Ch3}.
In fact, the loc. cit. implies that we can find an increasing sequence $\{ \lb_{-1,m} \}_m$, such that
$$
\lim_{m\to \infty} (\lb_{-1,m,\QQ})^{n-1}  = [K:\QQ]\dvol(\mb), \quad
\lim_{m\to \infty} (\lb_{-1,m})^{n} = \volhat(\mb).
$$
This finishes the proof of Theorem \ref{main3}.

\subsubsection*{Small case}

The proof of Theorem \ref{main4} in the nef case is very similar, except that we use Theorem \ref{main2} instead of Theorem \ref{main1} to bound $r_i$ in terms of $d_i$.
Note that we can assume every $\CX_i$ to have smooth generic fiber by a further generic resolution of singularities. We leave the details to interested readers.

\subsection{Arithmetic 3-folds}
In this section, we will prove Theorem \ref{3fold}. Let us resume the general setting. Here $\CX$ is an arithmetic 3-fold over $O_K$, and $\lb$ is a nef hermitian line bundle on $\CX$ such that
$$
\phi_{\CL_K}: \CX_K \dashrightarrow \mathbb P (H^0(\CL_K))
$$
is a generically finite rational map.

\subsubsection*{Linear series on algebraic surfaces}

Let $S$ be an algebraic surface over an algebraically closed field $k$. We always use $\kappa(S)$ to denote the Kodaira dimension of $S$.

Let $L$ be a line bundle
on $S$. Assume $h^0(L) > 1$. Hence we have the rational map
$$
\phi_L: S \dashrightarrow \mathbb P (H^0(L)).
$$
We say $\phi_L$ is \textit{generically finite} if $\dim \phi_L(S)=2$. Otherwise, $\dim \phi_L(S)=1$, and in this case we say $\phi_L$ is \textit{composed with a pencil}.

\begin{theorem} \cite[Theorem 1.2]{Sh} \label{Sh}
Assume that $\kappa(S)\geq 0$. If $L$ is nef and $\phi_L$ is generically finite, then
$$
h^0(L)\leq \frac12 L^2 +2.
$$
\end{theorem}

This result is cleaner than the surface case of Theorem \ref{main1}, under more assumptions.

\begin{theorem} \label{castelnuovo}
Assume that $\kappa(S) \ge 0$. Let $L$ and $M$ be line bundles on $\CX$.
Assume that $L$ is nef and $\phi_L$ is generically finite, and that $L-M$ is effective and $\phi_M$ is composed with a pencil. Then
$$
h^0(M)\leq \frac12 L M +1.
$$
If furthermore the pencil of $\phi_M$ is not elliptic or hyperelliptic, then
$$
h^0(M)\leq \frac13 L M +1.
$$
\end{theorem}

\begin{proof}
By the basic construction in the geometric case, after blowing-up, we can assume that $\phi_L$ and $\phi_M$ are actually morphisms.
Since $\phi_M$ is composed with a pencil, we can write
$$
M \sim_{\rm num} a F,
$$
where $F$ is a general member of the pencil and $a \ge h^0(M) - 1$. We have
$$
LM = aLF + LZ \ge (h^0(M) - 1) LF.
$$
Because $S$ is not birationally ruled, $F$ is not rational.  Hence $LF \ge 2$ since $L$ is base-point-free on $F$. Moreover, $LF \ge 3$ if $F$ is not elliptic or hyperelliptic.
\end{proof}

\subsubsection*{Proof of Theorem \ref{3fold}}
Now we proof Theorem \ref{3fold}.
Apply Theorem \ref{degred} to $(\CX_0,\lb_0)=(\CX ,\lb)$.
Resume the notations in the theorem. Then we have the quadruples
$$
(\CX_i, \lb_i, \eb_i, c_i), \quad i=0, \cdots, N.
$$

We first analyze the proof of Theorem \ref{main3}.
Proposition \ref{onestep} in this case gives
$$
\hsefhat(\lb) - \frac{\lb^3}{6} \le \sum_{i=0}^N \left(r_i-\frac{1}{2}d_i\right) c_i + 4r_0\log{r_0}+2r_0\log 3.
$$
If $\phi_{\CL_{i,K}}$ is generically finite, then Theorem \ref{Sh} gives
$$
r_i-\frac{1}{2}d_i \leq 2,
$$
which is exactly what we need in the proof.
However, this inequality fails if $\phi_{\CL_{i,K}}$ is not generically finite, in which case $d_i=0$. So we need to bound $r_i$ in this case by a different method.

Since $\phi_{\CL_K}$ is generically finite, we can find the biggest $j \in \{0,1,\cdots, N\}$ such that $\phi_{\CL_{j,K}}$ is generically finite. Then $\phi_{\CL_{i,K}}$ is not generically finite for $i=j+1,\cdots, N$. We will bound $r_i$ by a variant of $d_i$ for such $i$.

For any $i = j,j+1,\cdots, N$, denote
$$
d'_i = \mathcal \CL_{j, \QQ} \cdot \CL_{i, \QQ}.
$$
Here by abuse of notation, the right-hand side denotes the intersection of
$\CL_{i, \QQ}$ with the pull-back of $\CL_{j, \QQ}$ to $\CX_{i,\QQ}$.
In the following, we still use this method to simplify our notations.

\begin{prop} \label{onestep1}
We have
$$
\lb'^3_j \ge 2 \sum_{i=j+1}^{N} d'_i c_i.
$$
\end{prop}

\begin{proof}
For any $i \ge j$, recall the decomposition
$$
\pi^*_i \lb_i' = \lb_{i+1}  + \ob(c_{i+1})= \lb'_{i+1} + \eb_{i+1} + \ob(c_{i+1}).
$$
Hence,
\begin{eqnarray*}
\lb_j'\cdot \lb'^2_i
 \ge  \lb_j' \cdot (\lb'_{i+1} + \ob(c_{i+1}))^2
=  \lb_j' \cdot \lb_{i+1}'^2 + 2 d'_{i+1}c_{i+1}.
\end{eqnarray*}
Summing over $i=j, \cdots, N-1$, one finishes the proof.
\end{proof}

Now we are ready to finish the proof.
By Proposition \ref{onestep},
\begin{eqnarray*}
\lb_0^3 & \geq &  \lb_j'^3 + 3 \sum_{i=0}^{j} d_i c_i,  \\
\hsefhat(\lb_0) &\leq&
\sum_{i=0}^{N} r_{i} c_i +4r_0 \log r_0+ 2r_0 \log 3.
\end{eqnarray*}
Note that the first inequality concerns the filtration from $0$ to $j$, while the second inequality concerns the filtration from $0$ to $N$.

By Proposition \ref{onestep1}, the first inequality implies
\begin{eqnarray*}
\lb_0^3 & \geq &  3 \sum_{i=0}^{j} d_i c_i+ 2 \sum_{i=j+1}^{N} d'_i c_i.
\end{eqnarray*}
Then the difference gives
$$
\hsefhat(\lb_0) - \frac16 {\lb_0^3} \le
\sum_{i=0}^{j} \left(r_{i} - \frac 12 d_i \right) c_i
+\sum_{i=j+1}^{N} \left(r_{i} - \frac 13 d'_i \right) c_i + 4r_0 \log r_0+ 2r_0 \log 3.
$$

By assumption,  $\kappa(\CX_K)\geq 0$ and $\CX_K$ has no elliptic or hyperelliptic pencil. Theorem \ref{Sh} and Theorem \ref{castelnuovo} give
\begin{eqnarray*}
r_{i} - \frac 12 d_i &\le&  2[K: \QQ], \quad i=0,\cdots, j,\\
r_{i} - \frac 13 d'_i &\le&  [K: \QQ], \quad i=j+1,\cdots, N.
\end{eqnarray*}
Hence,
\begin{eqnarray*}
\hsefhat(\lb) - \frac{1}{6}\lb^3
& \le & 2 [K: \QQ] \sum_{i=0}^N c_i + 4 r_0 \log r_0 + 2r_0 \log 3.
\end{eqnarray*}
Apply Lemma \ref{sumci} again. We have
\begin{eqnarray*}
\hsefhat(\lb) - \frac{1}{6}\lb^3
&\leq& \frac {2}{d_K} \lb^n + 4 r_0 \log r_0 + 2 r_0 \log 3.
\end{eqnarray*}
Combining with
$$
\hhat(\lb) \le \hsefhat(\lb) + r_0 \log 3,
$$
the proof is complete.

\end{document}